\documentclass{amsart}

\usepackage{multicol}        

\usepackage{mathrsfs} 
\usepackage{color,float, 
subcaption}
\restylefloat{figure}

\makeindex             

\usepackage{graphicx}

\usepackage{amsmath,amssymb,amsthm}

\usepackage{latexsym,array,tensor,verbatim,hyperref}
\usepackage[mathscr]{eucal}
\usepackage{eucal} 

\newtheorem{theorem}{Theorem}[section]
\newtheorem{lemma}[theorem]{Lemma}
\newtheorem{proposition}[theorem]{Proposition}

\theoremstyle{remake}
\newtheorem{remark}{Remark}

\begin{document}

\title[Universal Blowup Rate of Rotational NLS]{Universal Upper Bound on The Blowup Rate of Nonlinear Schr\"odinger Equation with Rotation}

\author{Yi Hu, Christopher Leonard and Shijun Zheng}

\address[Yi Hu]{Department of Mathematical Sciences, Georgia Southern University,
Statesboro, GA 30460}
 \email{yihu@GeorgiaSouthern.edu}
\address[Christopher Leonard]{Department of Mathematics, North Carolina State University,
Raleigh, NC 27695}
\email{cleonar@ncsu.edu}
\address[Shijun Zheng]{Department of Mathematical Sciences, Georgia Southern University,
Statesboro, GA 30460}
 \email{szheng@GeorgiaSouthern.edu}

\maketitle

\begin{abstract}
In this paper, we prove a universal upper bound on the blowup rate of a focusing nonlinear Schr\"odinger equation
with  an angular momentum under a trapping harmonic potential,
assuming that the initial data  is radially symmetric in the weighted Sobolev space.
The nonlinearity is in the mass supercritical and energy subcritical regime.
Numerical simulations are also presented.
\end{abstract}

\section{Introduction}\label{sec:1} 
 Consider the focusing nonlinear Schr\"odinger equation (NLS)\index{nonlinear Schr\"odinger equation} with an angular momentum term in $\mathbb{R}^{1+n}$:
	\begin{align}\label{eq:nls_va}
	\begin{cases}
	iu_t=-\Delta u+Vu-\lambda|u|^{p-1}u+L_Au \\
	u(0,x)=u_0\in\mathscr{H}^1.
	\end{cases}
	\end{align}
Here $u=u(t,x):\mathbb{R}\times\mathbb{R}^n\to\mathbb{C}$  denotes the wave function,
$V(x):=\gamma^2|x|^2$\, ($\gamma>0$) is a trapping harmonic potential\index{harmonic potential} that confines the movement of particles,
and $\lambda$ is a positive constant indicating the self-interaction between particles is attractive. 
The nonlinearity has the exponent $1\le p<2^*-1$, where by convention $2^*:=\frac{2n}{n-2}$ if $n\ge 3$;  $\infty$ if $n=1,2$. 
The operator $L_Au:=iA\cdot\nabla u$ is the angular momentum term,\index{angular momentum}
where $A=Mx$ with $M=(M_{j,k})_{1\leq j,k\leq n}$ being an $n\times n$ real-valued skew-symmetric matrix,
i.e., $M=-M^T$.
It  generates a rotation in $\mathbb{R}^n$
in the sense that $e^{-itL_A}f(x)=f(e^{tM}x)$ for $(t,x)\in\mathbb{R}\times\mathbb{R}^n$.
The space $\mathscr{H}^1=\mathscr{H}^{1,2}$ denotes the weighted Sobolev space\index{Sobolev space} 
	\begin{align*}
	\mathscr{H}^{1,r}(\mathbb{R}^n)
	:=\left\{f\in L^r(\mathbb{R}^n):\nabla f, \ xf \in L^r(\mathbb{R}^n)\right\}
	\end{align*}
for $r\in(1,\infty)$,
and the endowed norm is given by
	\begin{align*}
	\Vert f\Vert_{\mathscr{H}^{1,r}}
	=\Vert\nabla f\Vert_r+\left\Vert xf\right\Vert_r+\Vert f\Vert_r\, ,
	\end{align*}
where $\lVert\cdot\rVert_r:=\lVert\cdot\rVert_{L^r}$ is the usual $L^r$-norm.
The linear Hamiltonian $H_{A,V}:=-\Delta+V+iA\cdot\nabla$
is essentially self-adjoint in $L^2$, 
whose eigenvalues are associated to the Landau levels as quantum numbers.

When $n=3$, equation \eqref{eq:nls_va} is also known as Gross-Pitaevskii equation, 
which models rotating Bose-Einstein condensation (BEC)\index{BEC} with attractive particle interactions
in a dilute gaseous ultra-cold superfluid.
The operator $L_A$ is usually denoted by $-\Omega\cdot L$,
where $\Omega=(\Omega_1,\Omega_2,\Omega_3)\in\mathbb{R}^3$ is a given angular velocity vector
and $L=-ix\wedge\nabla$.
In this case the skew-symmetric matrix $M$ is equal to
$\begin{pmatrix} 0 & -\Omega_3 & \Omega_2 \\ 
\Omega_3 & 0 & -\Omega_1 \\ 
-\Omega_2 & \Omega_1 & 0 \end{pmatrix}$.

Such  system as given in \eqref{eq:nls_va},  the rotational nonlinear Schr\"odinger equation\index{nonlinear Schr\"odinger equation} (RNLS) describing rotating particles in a harmonic trap has acquired significance in connection with optics, plasma, quantized vortices, superfluids, spinor BEC in theoretical and experimental physics
\cite{Af06,BaoCai15, BaoWMar05,  Gross61, MAHHWC99, ReZaStri01}.
Meanwhile, 
  mathematical study of the solutions to equation (1) have been conducted  in order to 
provide insight and rigorous understanding for  the  dynamical behaviors of such  wave-matter. 
For $\lambda\in \mathbb{R}$ and $1\leq p<1+\frac{4}{n-2}$, 
the local well-posedness results of equation (\ref{eq:nls_va}) were obtained in e.g.
\cite{AnMaSpar, BaHaHuZheng,HaoHsiaoLi1}, 
see also \cite{CazE88,De91,GaZ13a, Zheng} for the treatment in a general magnetic setting.
In the focusing case $\lambda>0$ and $p\ge 1+\frac{4}{n}$,  there exist  
solutions that blowup in finite time \cite{BaHuZheng,Car02a,We83, Zh00,Zh05}.  

Let $Q\in H^1$ be the unique positive, non-increasing and radial ground state\index{ground state} 
solution of the elliptic equation 
 	\begin{align}\label{eq:ground_state_solution} 
	-\Delta Q+Q-Q^p=0\,,
	\end{align} 
where $H^1$ denotes the usual Sobolev space.\index{Sobolev space}  
In the mass-critical case $p=1+\frac{4}{n}$, the paper \cite{BaHaHuZheng} showed that 
 $\Vert Q\Vert_2$ serves as the sharp threshold for blowup and global existence for equation \eqref{eq:nls_va}. 
Moreover, if $p=1+\frac{4}{n}$ and
$\Vert u_0\Vert_2$ is slightly greater than $\Vert Q\Vert_2$,  the paper
\cite{BaHuZheng} obtained  the exact blowup rate\index{blowup rate}  
$ \Vert \nabla u(t) \Vert_2 =(2 \pi)^{-1/2}{\Vert \nabla Q \Vert_{2}} \sqrt{\frac{\log \left| \log (T - t)  \right| }{T - t} }$ 
as $t\to T=T_{max}$.
The analogous results for the standard NLS 
were initially proven in \cite{We83} and \cite{MerRa05b}, where $A=V=0$.  
In \cite{BaHuZheng} we apply the so-called $\mathcal{R}$-transform method, which is a composite of the lens transform and a time-dependent rotation   that allows to convert \eqref{eq:nls_va} into the standard NLS.   
We would like to mention that the case  $A=0$, $V=\gamma^2 |x|^2$ was considered in \cite{Zh00,ZhuZh10}.
Also if the harmonic potential\index{harmonic potential} is repulsive, 
i.e., $V=-\gamma^2|x|^2$, 
there are similar blowup results for equation \eqref{eq:nls_va} without angular momentum,  see e.g. \cite{ZhuLi}.

The purpose of this article is to give a space-time universal upper bound 
on the blowup rate\index{blowup rate}  for the blowup solution to equation \eqref{eq:nls_va} with radial data in the mass-supercritical regime
 $p\in( 1+\frac{4}{n},1+\frac{4}{n-2})$.  Our main result is stated as follows.

\begin{theorem}\label{thm:upper_bound}
Let $n\geq 3$ and $1+\frac{4}{n}<p<1+\frac{4}{n-2}$,
or $n=2$ and $3<p<5$.
Let $u_0\in\mathscr{H}^1$ be radially symmetric,
and assume that the corresponding solution $u\in C([0,T),\mathscr{H}^1)$ blows up in finite time $T$.
Then
	\begin{align}\label{univ-bluprate}
	\int_t^T (T-s) \Vert\nabla u(s)\Vert_{2}^2 ds
	\leq C(T-t)^\frac{2(5-p)}{5-p+(n-1)(p-1)}.
	\end{align}
\end{theorem}

Theorem \ref{thm:upper_bound} is motivated by a similar result by Merle, Rapha\"el and Szeftel \cite{MerRaSzef}, 
where they proved such an upper bound on the blowup rate for the standard NLS without potential or angular momentum. 
The proof of Theorem \ref{thm:upper_bound} mainly follows  the idea in \cite{MerRaSzef} 
but  relies on a refined version of the localized virial identity\index{virial identity} 
(Lemma \ref{l:local-virial_gamma}, Section 3) in the magnetic setting. Note that 
 the $\mathcal{R}$-transform introduced in \cite{BaHuZheng} does not apply here for $p> 1+\frac{4}{n}$.
In Section \ref{section:proof_of_the_main_theorem} we shall give the proof of the main theorem. 
In Section \ref{num-figures}, we include numerical figures to show the threshold of blowup for various cases of interest.

\section{Preliminaries}

In this section we recall the local well-posedness  theory for equation \eqref{eq:nls_va} 
 and a radial version of Gagliardo-Nirenberg  inequality  that we shall apply in the proof of our main theorem.

\subsection{Local well-posedness of RNLS for $p\in [1, 1+4/(n-2) )$}
For $u_0\in\mathscr{H}^1$, 
the local well-posedness of equation \eqref{eq:nls_va} was obtained as a special case in e.g, 
\cite{De91,Zheng}
and \cite{BaHuZheng}. 
The papers \cite{De91,Zheng} dealt with a general class of magnetic potentials and  electric potentials where $A$ is  sublinear and $V$ is subquadratic and essentially of positive sign.   
The case where  $V$ is subquadratic of both signs, e.g., $V=\pm \sum_{j=1}^n \gamma_j^2 x_j^2$, 
$\gamma_j>0$ were considered in \cite{AnMaSpar}
when $n=2,3$, and  \cite{BaHuZheng,Car11time} in higher dimensions. 

Let $H_{A,V}=-\Delta+V+i A\cdot \nabla=-(\nabla-\frac{i}{2}A)^2+V_e$, where $V_e(x)=V(x)-\frac{|A|^2}{4}$ and $\mathrm{div}\, A=0$.
The proof for the local result relies on local in time dispersive estimates for the propagator 
$U(t)=e^{-itH_{A,V}}$  
constructed in \cite{Ya}. 
Alternatively, for $V=\pm \gamma^2 |x|^2$, this can also be done by means of
	\begin{align}\label{U(t)OmV}
	e^{-itH_{A,V}}(x,y)
	=\left(\frac{\gamma}{2\pi i\sin(2\gamma t)}\right)^\frac{n}{2}
	e^{i\frac{\gamma}{2}(|x|^2+|y|^2)\cot(2\gamma t)}
	e^{-i\gamma\frac{(e^{tM}x)\cdot y}{\sin(2\gamma t)}},
	\end{align}
the fundamental solution to $iu_t=H_{A,V}u$ 
  if $V(x)= \gamma^2 |x|^2$;
and replacing $\gamma\to i\gamma$ if   $V(x)= -\gamma^2 |x|^2$.
The above formula (\ref{U(t)OmV}) can be obtained via the $\mathcal{R}$-transform, a type of pseudo-conformal transform in the rotational setting, see \cite{BaHuZheng}.

\begin{proposition}\label{prop:wellposedness}
For equation \eqref{eq:nls_va},
we have the following known results on well-posedness and conservation laws.
Let $r:=p+1$ and $q:=\frac{4(p+1)}{n(p-1)}$\ .
	\begin{enumerate}
	\item[(a)] Well-posedness and blowup alternative:
		\begin{enumerate}
		\item[(i)] If $1\leq p<1+\frac{4}{n}$,
		then equation \eqref{eq:nls_va} has an $\mathscr{H}^1$-bounded global solution
		$u\in C(\mathbb{R};\mathscr{H}^1)\cap L^q_{\rm loc}(\mathbb{R};\mathscr{H}^{1,r})$.
	
		\item[(ii)] If $1+\frac{4}{n}\leq p<1+\frac{4}{n-2}$,
		then there exists $T=T_{max}>0$ such that equation \eqref{eq:nls_va} has a unique maximal solution
		$u\in C([0,T),\mathscr{H}^1)\cap L^q_{\rm loc}([0,T),\mathscr{H}^{1,r})$.
		If $T<\infty$,
		then $u$ blows up at $T$ with a lower bound
			\begin{align}\label{lowbd:T-t}
			\Vert\nabla u(t) \Vert_2
			\geq C(T-t)^{-(\frac{1}{p-1}-\frac{n-2}{4})}.
			\end{align}
		\end{enumerate}
	\item[(b)] The followings are conserved on the maximal lifespan $[0,T)$:
		\begin{enumerate}
		\item[(i)] Mass:
		$\displaystyle M[u]=\int|u|^2$
		\item[(ii)] Energy:
		$\displaystyle E[u]=\int\left(|\nabla u|^2+V|u|^2-\frac{2\lambda}{p+1}|u|^{p+1}+\overline{u}L_Au\right)$
		\item[(iii)] Angular momentum:\index{angular momentum}
		$\displaystyle \ell_A[u]=\int\overline{u}L_Au$.
		\end{enumerate}
	\end{enumerate}
\end{proposition} 
\begin{proof} 
Here we briefly outline  the proof. 
In virtue of \cite{Ya}, the kernel representation for $U(t)$ is given by
\begin{align}
 &U(t)f(x)=(2\pi it )^{-n/2}\int e^{iS(t,x,y)}a(t,x,y)f(y)dy\ , \label{ker-U(t)AV}
 \end{align} 
 where $S(t,x,y)$ is  real-valued in $C^\infty(I^*_\delta\times \mathbb{R}^{2n})$, 
 $I^*_\delta:=(-\delta,\delta)\setminus\{0\}$ for some positive constant $\delta$,
 and $a(t,x,y)$ is in $L^\infty\cap C^\infty(\mathbb{R}\times \mathbb{R}^{2n})$.
 Then, from (\ref{ker-U(t)AV}) it follows  the dispersive estimate for $0\ne |t|<\delta$, 
	\begin{align}
	\Vert U(t)f \Vert_{L^\infty} \lesssim \frac{1}{|t|^{n/2}} \Vert f \Vert_{L^1}\ .\label{L1-infty-disp}
	\end{align} 
This, together with $\left\Vert U(t)f\right\Vert_2=\left\Vert f\right\Vert_2$ 
yields the Strichartz type estimates on $I=I^*_\delta$ 
\begin{align} 
&\Vert U(t)f \Vert_{L^q(I, \mathscr{H}^{1,r})}\lesssim \Vert f\Vert_{\mathscr{H}^1}\label{stri-qr-H}\\
&\Vert  \int_0^t U(t-s)F(s,\cdot) ds\Vert_{L^q(I, \mathscr{H}^{1,r})}\lesssim C
 \Vert F \Vert_{L^{\tilde{q}'}(I, \mathscr{H}^{1,\tilde{r}'})}\label{inhom-stri-H}
\end{align} 
where  $(q,r)=(q,r,n)$, $(\tilde{q},\tilde{r})=(\tilde{q},\tilde{r},n)$ are  admissible pairs 
satisfying $q,r\in [2,\infty]$,  $(q,r,n)\neq (2,\infty,2)$ and
\begin{align*} 
\frac{2}{q}+\frac{n}{r}=\frac{n}{2}\ ,
\end{align*} 
with $q'$ denoting the H\"older conjugate of $q$. 
 Hence the local in time existence of (\ref{eq:nls_va}) holds.
If $p\in [1+{4}/{n},1+{4}/{(n-2)})$, the blowup alternative and the lower bound for blowup rate\index{blowup rate}  of (\ref{eq:nls_va}) follow from standard argument as in \cite{BaHuZheng}.
\end{proof}
\begin{remark}
The  Strichartz estimates (\ref{stri-qr-H})-(\ref{inhom-stri-H})  generalize  
 those obtained in  \cite{De91,Zheng} where  $V_e(x)\approx \beta |x|^2$, $\beta>0$ as $|x|\to \infty$.  
 Here we allow $V_e$ to be any quadratic function asymptotically 
 $V_e(x)\approx \sum \beta_{ij} x_ix_j$ with $\beta_{ij}\in \mathbb{R}$. 
In the proof of  (\ref{stri-qr-H})-(\ref{inhom-stri-H}), we directly study the action of $U(t-s)$ in the 
weighted space $\mathscr{H}^{1,r}$ 
based on \cite[Lemma 3.1]{Ya}, an oscillatory integral operator (OIO) formula of  Yajima in the magnetic setting.  The OIO method was initially applied by  Fujiwara in treating electric potentials. 
Our approach allows to technically deal with the commuting issue between $x$, $\nabla$ and $U(t-s)$, 
and provides a treatment for general sublinear $A$ and subquadratic $V$  
assumed in Proposition \ref{prop:wellposedness}. 
Special cases of $A$ and $V$ for the RNLS
were studied in the literature, see e.g., \cite{AnMaSpar,Car11time,CazE88,GaZ13a,HaoHsiaoLi1}. 
\end{remark}

Let $Q=Q_0$ be the ground state\index{ground state}  
solution of (\ref{eq:ground_state_solution}). 
In the $L^2$-critical case $p=1+{4}/{n}$, from \cite{BaHaHuZheng, BaHuZheng} we know that $\Vert Q_0\Vert_2$ is the sharp threshold in the sense that:
	\begin{enumerate}
	\item[(a)] If $\Vert u_0\Vert_2<\Vert Q_0\Vert_2$,
	then equation \eqref{eq:nls_va} has a unique global in time solution.
	\item[(b)] For all $c\geq \Vert Q_0\Vert_2$,
	there exists $u_0$ with $\Vert u_0\Vert_2=c$ so that $u$ is a finite time blowup solution of equation \eqref{eq:nls_va}.
	\end{enumerate} 
According to \cite[Proposition 4.5]{BaHuZheng}, 
if $p=1+4/n$ and
$\Vert u_0\Vert_2=\Vert Q_0\Vert_2$, then all blowup solutions of (\ref{eq:nls_va}) have the pseudo-conformal 
blowup rate 
\begin{align}\label{e:dU-pseudo}
&\Vert \nabla u\Vert_2 =O\left(  (T-t)^{-1}\right), \quad \text{as} \;t\rightarrow T.
\end{align}
The  case  $p>1+4/n$ are technically more challenging. As far as we know, 
 there have not been results on the characterization for the blowup profile or blowup rate. 
Theorem \ref{thm:upper_bound} provides an upper bound for the rotational NLS (\ref{eq:nls_va}) under a harmonic potential.\index{harmonic potential} 
For the standard NLS, such upper bound is sharp, 
which is shown by constructing a ring-blowup solution in  \cite{MerRaSzef}. However, we do not know 
if the estimate (\ref{univ-bluprate}) is sharp for equation (\ref{eq:nls_va}), since the $\mathcal{R}$ transform does not apply for the $L^2$-supercritical case.

\subsection{Radial Gagliardo-Nirenberg inequality}
The following is a radial version of Gagliardo-Nirenberg inequality due to \mbox{W.~A.~Strauss.}

\begin{lemma}\label{lemma:radial_GN}
Let $u\in {H}^1$ be a radial function.  Then for $R>0$,  there holds true
	\begin{align*}
	\Vert u\Vert_{L^\infty(|x|\geq R)}
	\leq C\frac{\Vert\nabla u\Vert_2^{1/2} \Vert u\Vert_2^{1/2}}{R^\frac{n-1}{2}}\,.
	\end{align*}
\end{lemma} 

To prove Lemma \ref{lemma:radial_GN}, 
first note that when $n=1$, 
the classical Gagliardo-Nirenberg inequality reads $\Vert u\Vert_{\infty}\leq C \Vert u'\Vert_2^{1/2} \Vert u\Vert_2^{1/2}$.
For general dimensions, 
since $u$ is radial,
we denote $u(x)=v(|x|)=v(r)$ and note that 
	\begin{align*}
	\Vert u\Vert_{2}
	\geq  \Vert u\Vert_{L^2(|x|\geq R)}
	=C\left(\int_R^\infty|v(r)|^2 r^{n-1}dr\right)^{1/2}
	\geq CR^\frac{n-1}{2}\left(\int_R^\infty|v(r)|^2dr\right)^{1/2}
	=CR^\frac{n-1}{2}\Vert v_R\Vert_2 \,,
	\end{align*}
where $v_R=v\vert_{\{r\geq R\}}$.
Similarly,
we have $\Vert \nabla u\Vert_2\geq CR^\frac{n-1}{2} \Vert v_R' \Vert_2$. 
Combining these with the above one-dimensional  inequality we obtain
	\begin{align*}
	\left\Vert u\right\Vert_{L^\infty(|x|\geq R)}
	=\Vert v_R\Vert_\infty
	\leq C\Vert v_R' \Vert_2^{1/2} \Vert v_R \Vert_2^{1/2}
	\leq C\frac{ \Vert\nabla u\Vert_2^{1/2} \Vert u\Vert_2^{1/2}}{R^\frac{n-1}{2}}\,.
	\end{align*}

\section{Localized virial identity} 

To prove  Theorem \ref{thm:upper_bound} we derive certain localized virial identity\index{virial identity} associated to equation (\ref{eq:nls_va}). 
 This type of identities were shown in \cite{MerRa} in the case $A=V=0$
 and in \cite{FV09,Gar12} for some general electromagnetic potentials.
Here we present a direct proof for $A=Mx$ ($M$  skew-symmetric) and general $V$, which is different than that in \cite{FV09, Gar12}. 
   Let $C^\infty_0=C^\infty_0(\mathbb{R}^n)$ denote the space of $C^\infty$ functions with compact support.

\begin{lemma}[Localized virial identity]\label{l:local-virial_gamma} 
Assume that $u\in C([0,T),\mathscr{H}^1)$ is a solution to equation \eqref{eq:nls_va}.
Define $\displaystyle J(t):=\int\varphi|u|^2$ for any real-valued radial  function $\varphi\in C^\infty_0$.
Then
	\begin{align}
	J'(t)
	=2\Im\int\overline{u}\nabla\varphi\cdot\nabla u\,, \label{eq:J'}
	\end{align}
and
	\begin{align}\label{eq:J''}
	\begin{split}
	J''(t)
	&=-\int\Delta^2\varphi|u|^2
	-\frac{2\lambda (p-1)}{p+1} \int\Delta\varphi|u|^{p+1}
	+4\int\left(\frac{\varphi''}{r^2}-\frac{\varphi'}{r^3}\right)|x\cdot\nabla u|^2 \\
	&\qquad+4\int\frac{\varphi'}{r}|\nabla u|^2
         -2\int \nabla \varphi \cdot\nabla V  |u|^2\,.
	\end{split}
	\end{align}
\end{lemma}

\begin{proof}
Note that
	\begin{align*}
	J'(t)
	&=\int\varphi u_t\overline{u}
	+\int\varphi u\overline{u}_t
	=2\Re\int\varphi\overline{u}u_t \\
	&=2\Re\int\varphi\overline{u}(i\Delta u-iVu+i\lambda|u|^{p-1}u+A\cdot\nabla u)
	:=2(I_1+I_2+I_3+I_4).
	\end{align*}
The term $I_1$ is estimated as
	\begin{align*}
	I_1
	=\Re\left(i\int\varphi\overline{u}\Delta u\right)
	=\Re\left(-i\int\varphi|\nabla u|^2-i\int\overline{u}\nabla\varphi\cdot\nabla u\right)
	=\Im\int\overline{u}\nabla\varphi\cdot\nabla u.
	\end{align*}
Obviously
$\displaystyle I_2=\Re\left(-i\int\varphi V|u|^2\right)=0$ and
$\displaystyle I_3=\Re\left(i\lambda\int\varphi|u|^{p+1}\right)=0$.
For $I_4$,
we have
	\begin{align*}
	I_4
	&=\Re\int\varphi\overline{u}A\cdot\nabla u
	=\Re\left(-\int\nabla\varphi\cdot A|u|^2
	-\int\varphi\nabla\overline{u}\cdot Au
	-\int\varphi (\nabla\cdot A) |u|^2\right) \\
	&=-\int\nabla\varphi\cdot A|u|^2
	-I_4
	-\int\varphi\nabla\cdot A|u|^2.
	\end{align*}
Since $\varphi$ is radial and $M$ is skew-symmetric, we know that (with $r=|x|$)
	\begin{align}\label{eq:radial_asymmetric_estimates}
	\nabla\varphi\cdot A=\varphi'(r)\frac{x}{r}\cdot A=\frac{\varphi'(r)}{r}x\cdot(Mx)=0 \qquad
	{\rm and} \qquad
	\nabla\cdot A=0.
	\end{align}
So $I_4=-I_4$ and this implies $I_4=0$.
Hence \eqref{eq:J'} follows.

Differentiating \eqref{eq:J'} again,
we have 
	\begin{align*}
	J''(t)
	&=2\left(\Im\int\overline{u}_t\nabla\varphi\cdot\nabla u
	+\Im\int\overline{u}\nabla\varphi\cdot\nabla u_t\right)
	=2\left(\Im\int\overline{u}_t\nabla\varphi\cdot\nabla u
	-\Im\int\nabla\cdot(\overline{u}\nabla\varphi)u_t\right) \\
	&=2\left(-\Im\int\Delta\varphi\overline{u}u_t
	-2\Im\int\nabla\varphi\cdot\nabla\overline{u}u_t\right)
	:=2(-S-2T).
	\end{align*}
To estimate $S$,
first we write
	\begin{align*}
	S
	=\Im\int \Delta\varphi \,\overline{u} u_t
	=\Im\int\Delta\varphi \,\overline{u}(i\Delta u-iVu+i\lambda|u|^{p-1}u+A\cdot\nabla u)
	:=S_1+S_2+S_3+S_4.
	\end{align*}
Since
	\begin{align*}
	S_1
	&=\Im\left(i\int\Delta\varphi\overline{u}\Delta u\right)\\
	&=\int\Delta^2\varphi|u|^2
	+\Im\left(-i\int\Delta\varphi\Delta\overline{u}u\right)
	-2\int\Delta\varphi|\nabla u|^2
	=\int\Delta^2\varphi|u|^2
	-S_1
	-2\int\Delta\varphi|\nabla u|^2,
	\end{align*}
we have $\displaystyle S_1=\frac{1}{2}\int\Delta^2\varphi|u|^2-\int\Delta\varphi|\nabla u|^2$.
Obviously,
$\displaystyle S_2=-\int\Delta\varphi V|u|^2$
and $\displaystyle S_3=\lambda\int\Delta\varphi|u|^{p+1}$. 
In $S_4$,
since $\Delta\varphi$ is also radial,
by \eqref{eq:radial_asymmetric_estimates} we note that
	\begin{align}\label{eq:pure_imaginary}
	\begin{split}
	\int\Delta\varphi\, \overline{u}A\cdot\nabla u
	&=-\int\nabla (\Delta\varphi) \cdot A|u|^2
	-\int\Delta\varphi\nabla\overline{u}\cdot Au
	-\int\Delta\varphi (\nabla\cdot A ) |u|^2 \\
	&=-\overline{\int\Delta\varphi \,\overline{u}A\cdot\nabla u}\,,
	\end{split}
	\end{align}
indicating $\displaystyle \int\Delta\varphi \,\overline{u}A\cdot\nabla u$ is  imaginary.
So $\displaystyle S_4=\Im\int\Delta\varphi\overline{u}A\cdot\nabla u=-i\int\Delta\varphi\overline{u}A\cdot\nabla u$.

To estimate $T$,
first we write
	\begin{align*}
	T
	=\Im\int\nabla\varphi\cdot\nabla\overline{u}(i\Delta u-iVu+i\lambda|u|^{p-1}u+A\cdot\nabla u)
	:=T_1+T_2+T_3+T_4.
	\end{align*}
For $T_1$,
one has
	\begin{align*}
	T_1
	&=\Im\left(i\int\sum_{j,k}\varphi_{x_j}\overline{u}_{x_j}u_{x_kx_k}\right)
	=\Im\left(-i\int\sum_{j,k}\varphi_{x_jx_k}\overline{u}_{x_j}u_{x_k}
	-i\int\sum_{j,k}\varphi_{x_j}\overline{u}_{x_jx_k}u_{x_k}\right)
	:=T_{1,1}+T_{1,2}.
	\end{align*}
Since $\varphi$ is radial,
we have $\displaystyle \varphi_{x_j}=\varphi'(r)\frac{x_j}{r}$ and
$\displaystyle \varphi_{x_jx_k}=\varphi''(r)\frac{x_kx_j}{r^2}+\varphi'(r)\frac{\delta_{jk}}{r}-\varphi'(r)\frac{x_jx_k}{r^3}$,
so
	\begin{align*}
	T_{1,1}
	=-\int\frac{\varphi''}{r^2}|x\cdot\nabla u|^2
	-\int\frac{\varphi'}{r}|\nabla u|^2
	+\int\frac{\varphi'}{r^3}|x\cdot\nabla u|^2.
	\end{align*}
Also,
	\begin{align*}
	T_{1,2}
	=\Im\left(i\int\Delta\varphi|\nabla u|^2
	+i\int\sum_{j,k}\varphi_{x_j}\overline{u}_{x_k}u_{x_kx_j}\right)
	=\int\Delta\varphi|\nabla u|^2
	-T_{1,2}\,,
	\end{align*}
which reveals $\displaystyle T_{1,2}=\frac{1}{2}\int\Delta\varphi|\nabla u|^2$.
For $T_2$,
one has
	\begin{align*}
	T_2
	&=\Im\left(-i\int\nabla\varphi\cdot\nabla\overline{u}Vu\right)
	=\int\Delta\varphi V|u|^2
	+\int\nabla\varphi\cdot\nabla V|u|^2
	-T_2\,,
	\end{align*}
so $\displaystyle T_2=\frac{1}{2}\int\Delta\varphi V|u|^2+\frac12\int \nabla\varphi \cdot \nabla V |u|^2$. 
For $T_3$, there is
	\begin{align*}
	T_3
	&=\Im\left(i\lambda\int\nabla\varphi\cdot\nabla\overline{u}|u|^{p-1}u\right)
	=-\lambda\int\Delta\varphi|u|^{p+1} -pT_3\,,  
	\end{align*}
and so $\displaystyle T_3=-\frac{\lambda}{p+1}\int\Delta\varphi|u|^{p+1}$. 
For $T_4$,
we have
	\begin{align*}
	T_4
	=\Im\int (\nabla\varphi\cdot\nabla\overline{u}) \ (A\cdot\nabla u)
	=\Im\left(-\int\overline{u}\Delta\varphi A\cdot\nabla u
	-\int\overline{u}\nabla\varphi\cdot\nabla(A\cdot\nabla u)\right)
	:=T_{4,1}+T_{4,2}.
	\end{align*}
By \eqref{eq:pure_imaginary} we obtain $\displaystyle T_{4,1}=i\int\overline{u}\Delta\varphi A\cdot\nabla u$.
Also,
	\begin{align*}
	T_{4,2}
	&=\Im\left(-\int\overline{u}\sum_{j}\varphi_{x_j}\left(\sum_{k,l}M_{k,l}x_lu_{x_k}\right)_{x_j}\right) \\
	&=\Im\left(-\int\overline{u}\sum_{j,k,l}\varphi_{x_j}M_{k,l}\delta_{lj}u_{x_k}
	-\int\overline{u}\sum_{j,k,l}\varphi_{x_j}M_{k,l}x_lu_{x_kx_j}\right) \\
	&=\Im\left(-\int\overline{u}\sum_{j,k}\varphi_{x_j}M_{k,j}u_{x_k}
	+\int\sum_{j,k,l}\overline{u}_{x_k}\varphi_{x_j}M_{k,l}x_lu_{x_j}\right. \\
	&\hspace{.5in}\left.+\int\overline{u}\sum_{j,k,l}\varphi_{x_jx_k}M_{k,l}x_lu_{x_j}
	+\int\overline{u}\sum_{j,k,l}\varphi_{x_j}M_{k,l}\delta_{lk}u_{x_j}\right) \\
	&:=T_{4,2,1}+T_{4,2,2}+T_{4,2,3}+T_{4,2,4}.
	\end{align*}
Obviously $T_{4,2,2}=-T_4$,
and the skew-symmetry of $M$ implies $T_{4,2,4}=0$.
Note that
	\begin{align*}
	T_{4,2,3}
	&=\Im\left(-\int\sum_{j,k,l}\overline{u}_{x_j}\varphi_{x_k}M_{k,l}x_lu_{x_j}
	-\int\overline{u}\sum_{j,k,l}\varphi_{x_k}M_{k,l}\delta_{lj}u_{x_j}
	-\int\overline{u}\sum_{j,k,l}\varphi_{x_k}M_{k,l}x_lu_{x_jx_j}\right) \\
	&=-\Im\int A\cdot\nabla\varphi|\nabla u|^2
	-T_{4,2,1}
	-\Im\int\overline{u}A\cdot\nabla\varphi\Delta u=-T_{4,2,1} \, .
	\end{align*}
Hence $T_{4,2}=-T_4$ and so $\displaystyle T_4=\frac{i}{2}\int\overline{u}\Delta\varphi A\cdot\nabla u$.
Finally we obtain \eqref{eq:J''} by collecting all estimates on $S$'s and $T$'s.
\end{proof}

\section{Proof of the main theorem}\label{section:proof_of_the_main_theorem}

Now  we are ready to prove Theorem \ref{thm:upper_bound}.

\begin{proof} of Theorem \ref{thm:upper_bound}.
For a radial data $u_0$,
let $u$ be a corresponding radial solution that blows up in finite time $T<\infty$.
Then $x\cdot\nabla u=ru'$ and $|\nabla u|=|u'|$,
and the localized virial identity\index{virial identity} \eqref{eq:J''} can be written as, with $V=\gamma^2 |x|^2$
	\begin{align}\label{eq:J''_alternative}
	J''(t)
	&=-\int\Delta^2\varphi|u|^2
	-\frac{2\lambda(p-1)}{p+1}\int\Delta\varphi|u|^{p+1}
	+4\int\varphi''|\nabla u|^2
	-4\gamma^2\int x\cdot\nabla\varphi|u|^2.
	\end{align}
Choose a smooth radial function $\psi$ such that $\psi(x)=\frac{|x|^2}{2}$ if $|x|\leq 2$ and $\psi(x)=0$ if $|x|\geq 3$.
Pick a time $0<\tau<T$ and a radius $0<R=R(\tau)\ll 1$ (to be determined later).
Let $\varphi(x)=R^2\psi(\frac{x}{R})$.
Then, with $r = |x|$
	\begin{align*}
	\nabla\varphi(x)=R\psi'(\frac{r}{R})\frac{x}{r}\,, \quad
	\varphi''(r)=\psi''(\frac{r}{R})\,, \quad
	\Delta\varphi(x)=\Delta\psi(\frac{x}{R})\,, \quad
	\Delta^2\varphi(x)=\frac{1}{R^2}\Delta^2\psi(\frac{x}{R})\,,
	\end{align*}
so
	\begin{align*}
	J''(t)
	&=-\frac{1}{R^2}\int\Delta^2\psi(\frac{x}{R})|u|^2
	-\frac{2(p-1)}{p+1}\lambda\int\Delta\psi(\frac{x}{R})|u|^{p+1} \\
	&\qquad+4\int\psi''(\frac{r}{R})|\nabla u|^2
	-4\gamma^2R\int\psi'(\frac{r}{R})r|u|^2 \\
	&:=J_1+J_2+J_3+J_4.
	\end{align*}
Since $\Delta^2\psi$ is bounded,
and $\Delta^2\psi(\frac{x}{R}) = 0$ if $|x| \leq 2R$ or $|x| \geq 3R$,
we have $\displaystyle J_1\leq\frac{C}{R^2}\int_{2R\leq|x|\leq 3R}|u|^2$.
Also,
since $\Delta\psi$ is bounded,
and $\Delta\psi(\frac{x}{R})=n$ when $|x|\leq 2R$,
there is
	\begin{align*}
	J_2
	&=-\frac{2n\lambda(p-1)}{p+1}\int_{|x|\leq 2R}|u|^{p+1}
	-\frac{2\lambda(p-1)}{p+1}\int_{|x|>2R}\Delta\psi(\frac{x}{R})|u|^{p+1} \\
	&=-\frac{2n\lambda(p-1)}{p+1}\int|u|^{p+1}
	+\frac{2n\lambda(p-1)}{p+1}\int_{|x|>2R}|u|^{p+1}
	-\frac{2\lambda(p-1)}{p+1}\int_{|x|\geq 2R}\Delta\psi(\frac{x}{R})|u|^{p+1} \\
	&\leq-\frac{2n\lambda(p-1)}{p+1}\int|u|^{p+1}
	+C\int_{|x|\geq 2R}|u|^{p+1}.
	\end{align*}
By choosing $\psi$ such that $\psi''\leq 1$,
we have $\displaystyle J_3\leq 4\int|\nabla u|^2$.
And last,
since $\psi'(\frac{r}{R})=\frac{r}{R}$ when $|x|\leq 2R$ and $\psi'(\frac{r}{R})=0$ when $|x|\geq 3R$,
there is
	\begin{align*}
	J_4
	&=-4\gamma^2R\int_{|x|\leq 2R}\frac{r}{R}r|u|^2
	-4\gamma^2R\int_{2R<|x|\leq 3R}\psi'(\frac{r}{R})r|u|^2 \\
	&\leq-4\gamma^2\int_{|x|\leq 2R}|x|^2|u|^2
	+CR\int_{2R<|x|\leq 3R}|x||u|^2
	\leq CR\int_{2R<|x|\leq 3R}|x||u|^2.
	\end{align*}
Collecting all these terms,
we have
	\begin{align*}
	J''(t)
	&\leq 4\int|\nabla u|^2
	-\frac{2n\lambda(p-1)}{p+1}\int|u|^{p+1} \\
	&\qquad+C\left(\frac{1}{R^2}\int_{2R\leq|x|\leq 3R}|u|^2
	+\int_{|x|\geq 2R}|u|^{p+1}
	+R\int_{2R\leq|x|\leq 3R}|x||u|^2\right).
	\end{align*}
Recall that the energy $E[u]$ is conserved,
and $\ell_A[u]=0$ since $u$ is radial and $L_Au=0$,
we obtain
	\begin{align*}
	\int|u|^{p+1}
	=\frac{p+1}{2\lambda}\int|\nabla u|^2
	+\frac{p+1}{2\lambda}\int V|u|^2
	-\frac{p+1}{2\lambda}E[u_0]\,,
	\end{align*}
so
	\begin{align*}
	4\int|\nabla u|^2
	-\frac{2n\lambda(p-1)}{p+1}\int|u|^{p+1}
	=n(p-1)E[u_0]
	-(n(p-1)-4)\int|\nabla u|^2
	-n(p-1)\int V|u|^2.
	\end{align*}
This yields
	\begin{align*}
	J''(t)
	&\leq n(p-1)E[u_0]
	-(n(p-1)-4)\int|\nabla u|^2
	-n(p-1)\int V|u|^2 \\
	&\qquad+C\left(\frac{1}{R^2}\int_{2R\leq|x|\leq 3R}|u|^2
	+\int_{|x|\geq 2R}|u|^{p+1}
	+R\int_{2R\leq|x|\leq 3R}|x||u|^2\right).
	\end{align*}
Since $R\ll 1$,
we know $\frac{1}{R^2}\gg 3R^2$,
so
	\begin{align*}
	(n(p-1)-4)\int|\nabla u|^2
	+J''(t)
	\leq n(p-1)E[u_0]
	+C\left(\frac{1}{R^2}\int_{2R\leq|x|\leq 3R}|u|^2
	+\int_{|x|\geq 2R}|u|^{p+1}\right).
	\end{align*}
Also recall conservation of mass and $\frac{1}{R^2}\gg 1$,
we have
	\begin{align*}
	(n(p-1)-4)\int|\nabla u|^2
	+J''(t)
	&\leq n(p-1)E[u_0]
	+C\left(\frac{1}{R^2}
	+\int_{|x|\geq 2R}|u|^{p+1}\right) \\
	&\leq C\left(\frac{1}{R^2}
	+\int_{|x|\geq 2R}|u|^{p+1}\right)
	\end{align*}
To control the last term in the above inequality,
we apply Lemma \ref{lemma:radial_GN} and again the conservation of mass to obtain
	\begin{align*}
	\int_{|x|\geq 2R}|u|^{p+1}
	&\leq \Vert u\Vert_{L^\infty(|x|\geq 2R)}^{p-1} \int_{|x|\geq 2R} |u|^2 
	\leq C\frac{ \Vert\nabla u \Vert_2^\frac{p-1}{2}}{R^\frac{(n-1)(p-1)}{2}} \\
	&=\left[\left(\delta\frac{4n(p-1)-16}{p-1}\right)^\frac{p-1}{4} \Vert\nabla u \Vert_2^\frac{p-1}{2}\right]
	\frac{C}{R^\frac{(n-1)(p-1)}{2}} \\
	&\leq\delta(n(p-1)-4) \Vert\nabla u \Vert_2^2
	+\frac{C}{R^\frac{2(n-1)(p-1)}{5-p}}\,,
	\end{align*}
where the last inequality is an application of Young's inequality with $\frac{1}{\frac{4}{p-1}}+\frac{1}{\frac{4}{5-p}}=1$.
By choosing $\delta > 0$ small enough,
and noting that $\frac{2(n-1)(p-1)}{5-p}>2$,
we obtain
	\begin{align*}
	\frac{n(p-1)-4}{2} \Vert\nabla u\Vert_2^2
	+ J''(t)
	\leq\frac{C}{R^\frac{2(n-1)(p-1)}{5-p}}\,.
	\end{align*}
Integrate the above inequality over $[\tau,t]$ to obtain
	\begin{align*}
	\frac{n(p-1)-4}{2}\int_\tau^t \Vert\nabla u(s)\Vert_2^2
	+ J'(t)
	\leq\frac{C(t-\tau)}{R^\frac{2(n-1)(p-1)}{5-p}}
	+J'(\tau).
	\end{align*}
Recalling \eqref{eq:J'} and integrating the inequality again with respect to $t$ over $[\tau,t_0]$,
we have
	\begin{align*}
	&\frac{n(p-1)-4}{2}\int_\tau^{t_0}(t_0-s) \Vert\nabla u(s)\Vert_2^2
	+\int\varphi |u(t_0)|^2 \\
	&\leq C\frac{(t_0-\tau)^2}{R^\frac{2(n-1)(p-1)}{5-p}}
	+2(t_0-\tau)\Im\int\overline{u}(\tau)\nabla\varphi\cdot\nabla u(\tau) 
	+\int\varphi|u(\tau)|^2.
	\end{align*}
Recall that $\varphi(x)=R^2\psi(\frac{x}{R})$,
so
	\begin{align*}
	&\int_\tau^{t_0}(t_0-s)\Vert\nabla u(s)\Vert_2^2
	+\int\varphi|u(t_0)|^2 \\
	&\leq C\left(\frac{(t_0-\tau)^2}{R^\frac{2(n-1)(p-1)}{5-p}}
	+R(t_0-\tau)\left|\Im\int\overline{u}(\tau)\nabla\psi(\frac{x}{R})\cdot\nabla u(\tau)\right|
	+R^2\int|u(\tau)|^2\right) \\
	&\leq C\left(\frac{(t_0-\tau)^2}{R^\frac{2(n-1)(p-1)}{5-p}}
	+R(t_0-\tau)\left\Vert\overline{u}(\tau)\nabla\psi(\frac{\cdot}{R})\right\Vert_2\Vert\nabla u(\tau)\Vert_2
	+R^2 \Vert u_0\Vert_2\right) \\
	&\leq C\left(\frac{(t_0-\tau)^2}{R^\frac{2(n-1)(p-1)}{5-p}}
	+R(t_0-\tau)\Vert\nabla u(\tau)\Vert_2
	+R^2\right).
	\end{align*}
Letting $t_0\rightarrow T$ and applying Young's inequality  yield
	\begin{align*}
	\int_{\tau}^T(T-s) \Vert\nabla u(s)\Vert_2^2
	&\leq C\left(\frac{(T-\tau)^2}{R^\frac{2(n-1)(p-1)}{5-p}}
	+R(T-\tau) \Vert\nabla u(\tau)\Vert_2
	+R^2\right) \\
	&\leq C\left(\frac{(T-\tau)^2}{R^\frac{2(n-1)(p-1)}{5-p}}
	+R^2\right)
	+(T-\tau)^2 \Vert\nabla u(\tau)\Vert_2^2\,.
	\end{align*}
By setting $\displaystyle \frac{(T-\tau)^2}{R^\frac{2(n-1)(p-1)}{5-p}}=R^2$,
i.e., choosing $R=(T-\tau)^\frac{5-p}{5-p+(n-1)(p-1)}$,
we have
	\begin{align*}
	\int_\tau^{T}(T-s) \Vert\nabla u(s)\Vert_2^2
	\leq C(T-\tau)^\frac{2(5-p)}{5-p+(n-1)(p-1)}
	+(T-\tau)^2\Vert\nabla u(\tau)\Vert_2^2.
	\end{align*}
To solve this inequality,
let $\displaystyle g(\tau)=\int_{\tau}^{T}(T-s) \Vert\nabla u(s) \Vert_2^2$.
Then the above inequality becomes
	\begin{align*}
	g(\tau)
	\leq C(T-\tau)^\frac{2(5-p)}{5-p+(n-1)(p-1)}
	-(T-\tau)g'(\tau),
	\end{align*}
which is equivalent to
	\begin{align*}
	\frac{d}{d\tau}\left(\frac{g(\tau)}{T-\tau}\right)
	\leq C(T-\tau)^{-\frac{2(n-1)(p-1)}{5-p+(n-1)(p-1)}}.
	\end{align*}
Integrating this with respect to $\tau$ over $[0,t]$ yields $\displaystyle g(t)\leq C(T-t)^\frac{2(5-p)}{5-p+(n-1)(p-1)}$,
the desired result in (\ref{univ-bluprate}).
\end{proof}

\begin{remark}
Theorem \ref{thm:upper_bound} is still valid if we replace the positive constant $\lambda$
in equation \eqref{eq:nls_va} by a $C^1$-function $\lambda(x)$ that satisfies the following three conditions:
	\begin{enumerate}
	\item[(a)] There exist $\lambda_1,\lambda_2>0$ such that $\lambda_1\leq\lambda(x)\leq\lambda_2$;
	\item[(b)] $x\cdot\nabla\lambda\leq 0$;
	\item[(c)] $\nabla\lambda$ is bounded.
	\end{enumerate}
\end{remark} 
The proof proceeds the same way as that given in this section, but requires a version of 
Lemma \ref{l:local-virial_gamma} for the inhomogeneous NLS with rotation. We omit the details here. 

\begin{remark} Assume $T=T_{max}<\infty$ is the blowup time for the solution $u$ of (\ref{eq:nls_va}). 
Then (\ref{univ-bluprate}) implies that
 \begin{align*}
 \displaystyle \liminf_{t\rightarrow T} \;(T-t)^\delta \Vert\nabla u\Vert_2<\infty\, ,
 \end{align*}
where we note that the function $\delta:=\delta(p,n)=\frac{(n-1)(p-1)}{5-p+(n-1)(p-1)} \in \left(\frac{1}{2}, \frac{n-1}{2n-4}\right)$
is increasing in both $p$ and  $n$, given $p\in [1+4/n, 1+{4}/(n-2)) $.
 From (\ref{lowbd:T-t})
 we know that
  for any
  initial data in $\mathscr{H}^1$ one can  derive a general lower bound for the collapse rate, namely,
 there exists  $C=C_{p,n}> 0$ such that
	\begin{align*}
	\Vert\nabla u(t) \Vert_{2}
	\geq C (T - t)^{-(\frac{1}{p-1} - \frac{n-2}{4})}.
	\end{align*} 
In particular, if $p=1+4/n$, the estimate (\ref{univ-bluprate}) is only valid for the lower bound  $(T-t)^{-1/2}$. 
Thus, comparing the mass-critical case, where the $\log$-$\log$ law and pseudo-conformal blowup rate 
(\ref{e:dU-pseudo}) can occur,  the mass-supercritical case for larger data
can be more subtle, see \cite[Theorem 1.1]{BaHuZheng} and \cite{MerRaSzef}. 
\end{remark}

\section{Numerical results for mass-critical and mass-supercritical RNLS in 2D}\label{num-figures}
In this section we show numerical simulations for the blowup of \eqref{eq:nls_va}  with $n=2$ with given initial data $\psi_0$
being a multiple of the ground state\index{ground state} for the following nonlinear Schr\"odinger equation 
\index{nonlinear Schr\"odinger equation}
\begin{align}
    i\psi_t=-\frac{1}{2}\Delta \psi+\frac{1}{2}(\gamma_1^2x^2+\gamma^2_2 y^2)\psi -\lambda |\psi|^{p-1}\psi-i\Omega(y\partial_x-x\partial_y)\psi\, . \label{main2D}
\end{align} 
Let $Q=Q_{\Omega,V}$ be the ground state\index{ground state} for \eqref{main2D} satisfying the associated Euler-Lagrange equation 
\begin{align}\label{Q_OmV^p}
&  \omega Q=-\frac{1}{2}\Delta Q+\frac{1}{2}(\gamma_1^2x^2+\gamma^2_2 y^2)Q -\lambda |Q|^{p-1}Q-i\Omega(y\partial_x-x\partial_y)Q\,,
\end{align}
where $\omega$ is the chemical potential. 
 The construction of the ground states can be found e.g., in \cite{BaHaHuZheng,ELion89} if $p\le 1+4/n$.   
Here we use GPELab  as introduced in \cite{AntDub15} to do  the computations and  
 observe the blowup phenomenon for $\psi_0=C Q_{\Omega,V}$ with appropriate constant $C$
 for $p=3, 4$ and $p=6$. Note that the case $p=6$ is beyond the limit of exponents covered in Theorem \ref{thm:upper_bound}.  
For certain convenience from the software, we compute the solution $\psi$ of equation (\ref{main2D}) rather than  (\ref{eq:nls_va}) on the $(x,y)$-domain $[-3,3]\times [-3,3]$ in the plane. 
There is an obvious scaling relation between $\psi$ and $u$ of these two equations. 
From Subsection 2.1 we know that when $p=3$, the mass of $Q_{0,0}$ is the dichotomy
that distinguishes  the blow-up vs. global existence solutions.  
The main reason we use $Q_{\Omega,V}$ in place of $Q_{0,0}$ is that numerically the actual ground state $Q_{\Omega,V}$ is easier to compute and save as a stable profile under a trapping potential.

\begin{enumerate}
\item Isotropic case: $\gamma_1=\gamma_2=1$, $\lambda = 1$, $\Omega = 0.5$.  
Let $p=3$. We see in Figure \ref{fig:p3-Q(.5;11)}
 that
 the solution has energy concentration in short time and blows up with $\psi_0=2.5 Q_{\Omega,V}$,
 but it shows stable smooth solution at the level $\psi_0=2 Q_{\Omega,V}$.
For $p=4$, we observe in Figure \ref{fig:p4-V(11)} that 
  using $\psi_0= 2Q_{\Omega,V}$   yields a blowup solution;
but  there
 shows no blowup
at $1.6Q_{\Omega,V}$.

\begin{figure}[H]
\centering
\begin{subfigure}[b]{.45\textwidth}
	\includegraphics[width=\textwidth]{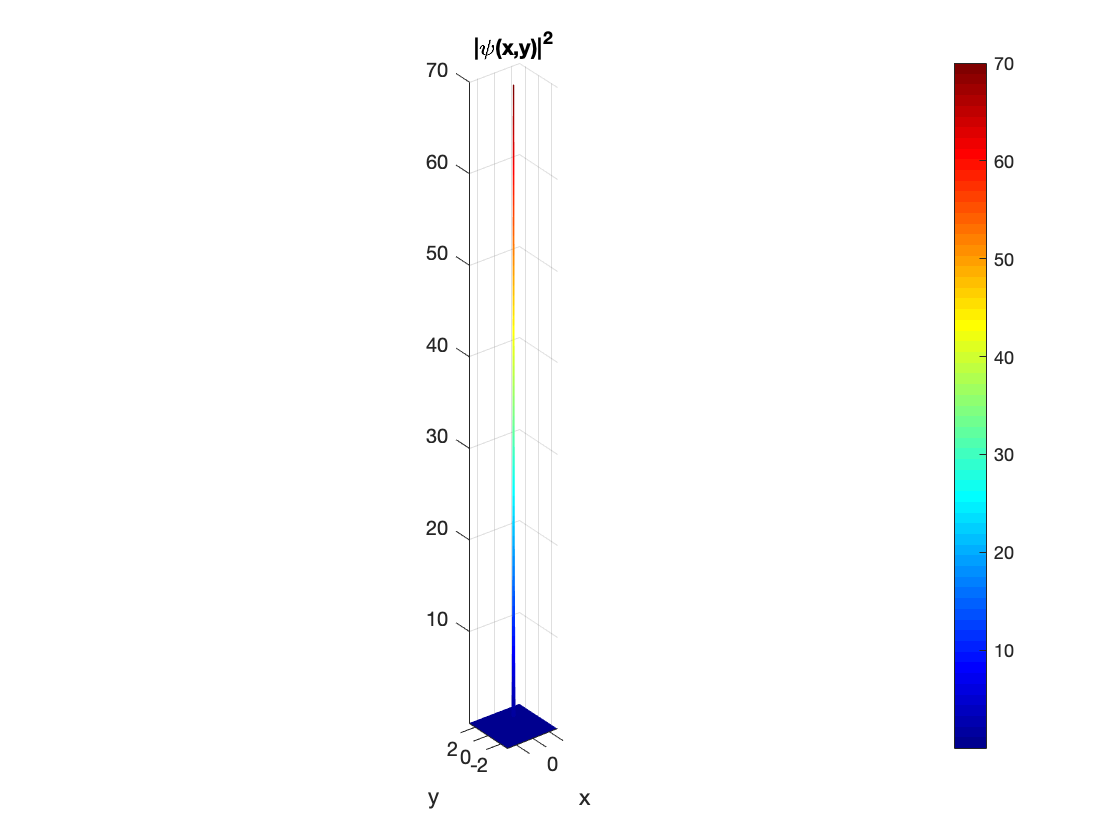}
	\caption{$\psi_0=2.5*Q_{\Omega,V}$ \;
	  (max $|\psi|^2\approx 1812)$}    
\end{subfigure}
\begin{subfigure}[b]{.45\textwidth}
\centering
	\includegraphics[width=\textwidth]{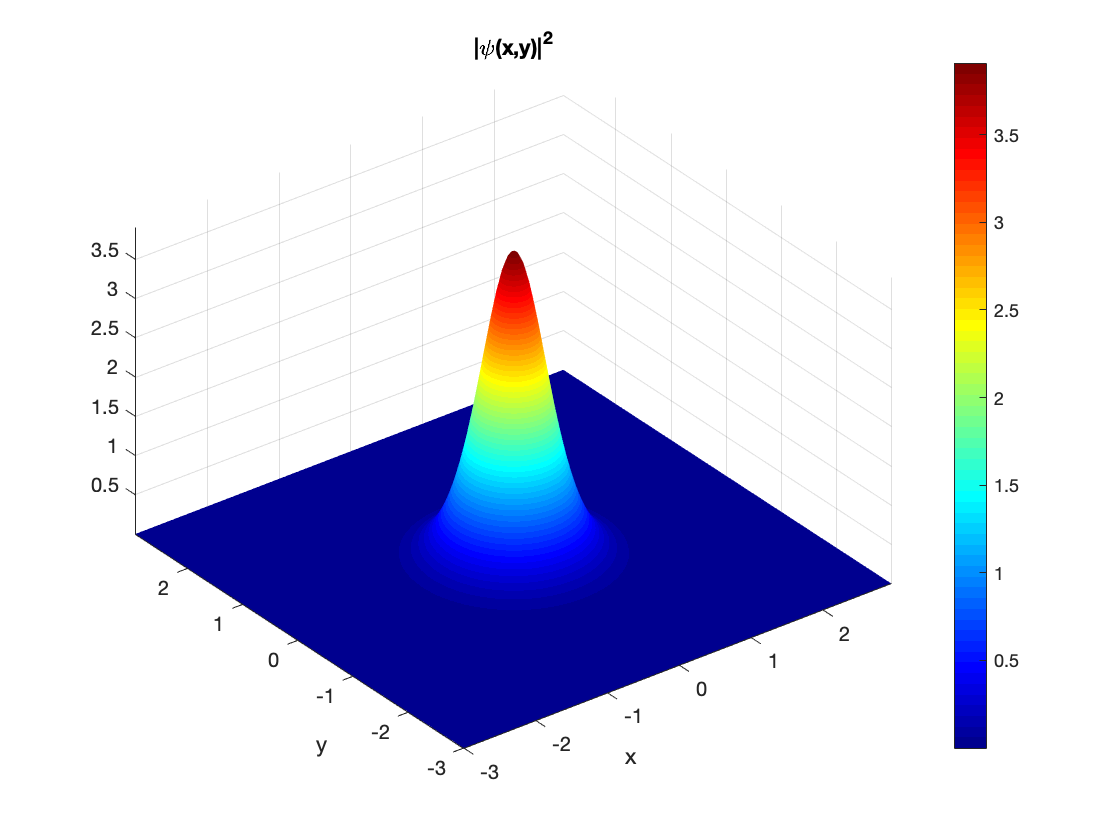}
	\caption{$\psi_0=2*Q_{\Omega,V}$\; (max $|\psi|^2\approx 3.9)$}    
\end{subfigure} 
\caption{$|\psi|^2$ when $p=3$, $(\gamma_1,\gamma_2)=(1,1)$, $\Omega=0.5$}
\label{fig:p3-Q(.5;11)}
\end{figure}

\begin{figure}[H] 
\centering
	\begin{subfigure}[b]{.45\textwidth}
	\centering
	\includegraphics[width=\textwidth]{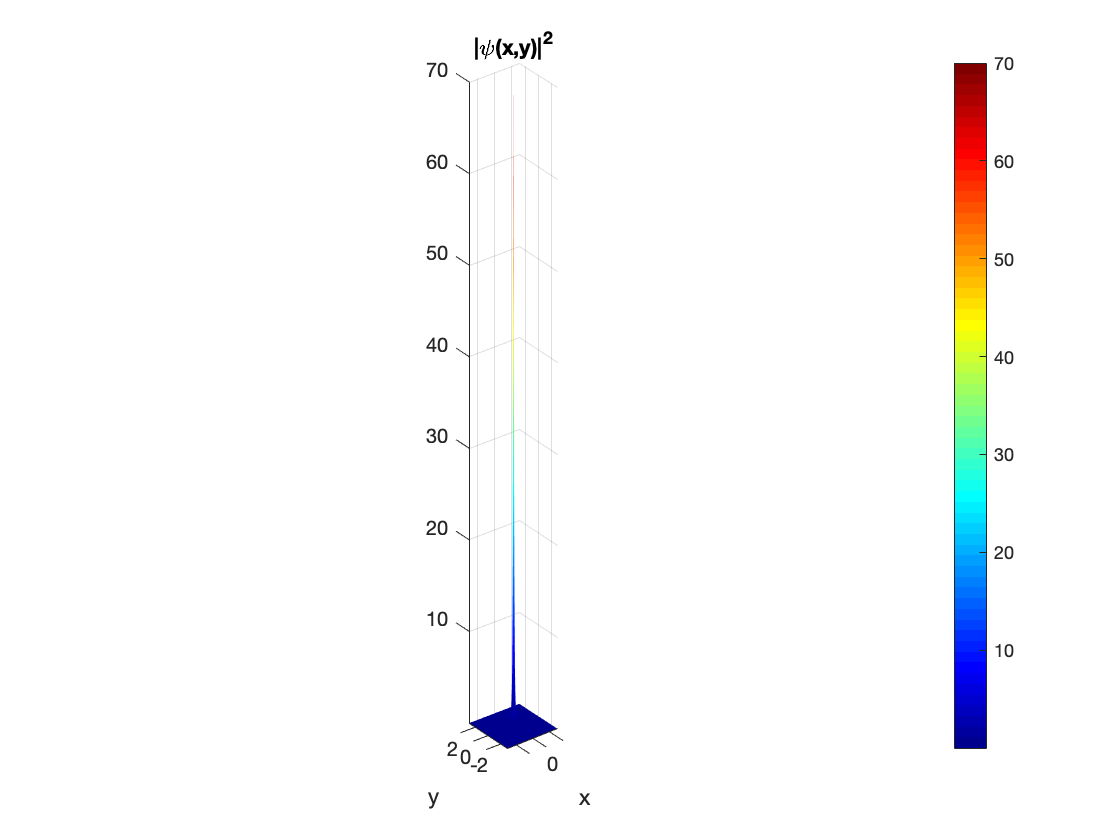}
	\caption{$\psi_0=2*Q_{\Omega,V}$\;
	   (max $|\psi|^2\approx 102)$)}    
\end{subfigure}
\begin{subfigure}[b]{.45\textwidth}
\centering 
	\includegraphics[width=\textwidth]{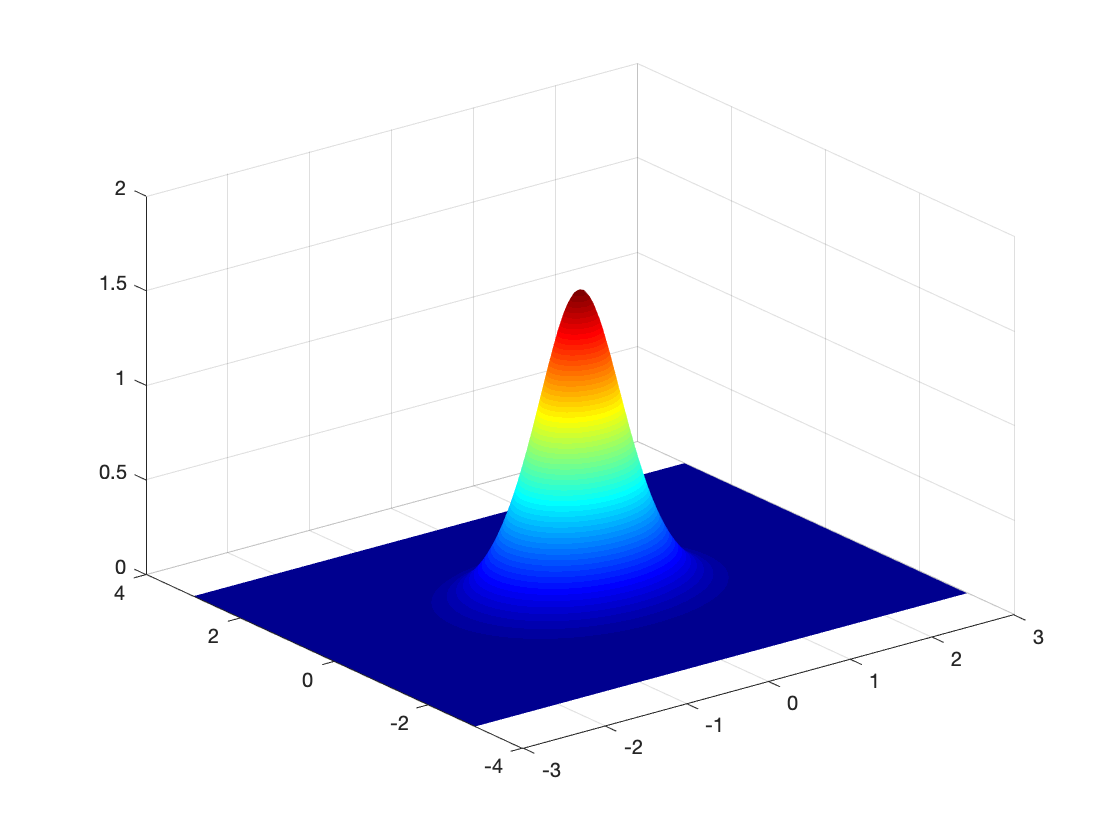} 
	\caption{$\psi_0=1.6*Q_{\Omega,V}$\; (max $|\psi|^2\approx 1.63)$. }    
\end{subfigure}
\caption{$|\psi|^2$ when $p=4$, $(\gamma_1,\gamma_2)=(1,1)$, $\Omega=0.5$}
\label{fig:p4-V(11)}
\end{figure}

\item {Anisotropic case: $\gamma_1=1$, $\gamma_2=2$, $\lambda = 1$, $\Omega = 0.5$.} Let $p=4$.
We observe that the anisotropic harmonic potential\index{harmonic potential}  may yield blowup at a lower level ground state.\index{ground state} 
Figure \ref{fig:p4-V12} shows blowup
when $\psi_0=1.8Q_{\Omega,V}$;
while stable smooth solution at $\psi_0=1.5Q_{\Omega,V}$.

\begin{figure}[H]
\centering
\begin{subfigure}[b]{.45\textwidth}
	\centering
	\includegraphics[width=\textwidth]{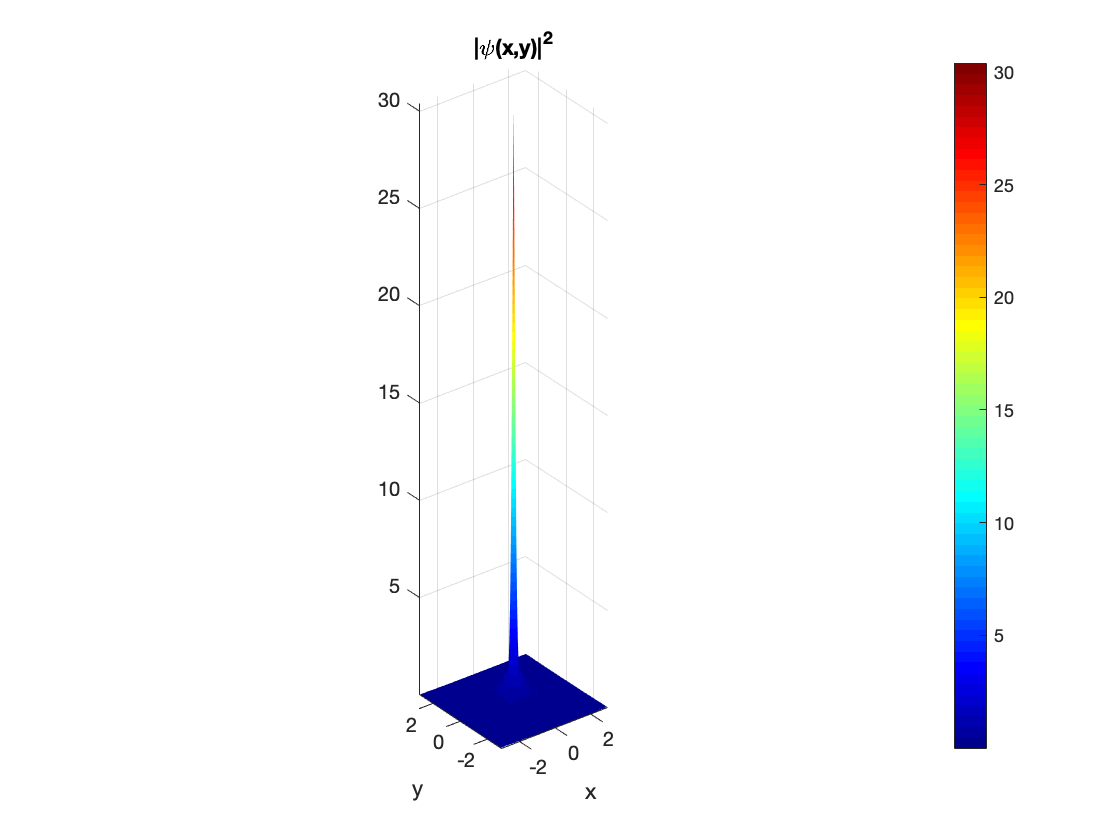}
	\caption{$\psi_0=1.8*Q_{\Omega,V}$\; 
	   (max $|\psi|^2\approx 93)$}    
\end{subfigure} 
\begin{subfigure}[b]{.45\textwidth}
       \centering
	\includegraphics[width=\textwidth]{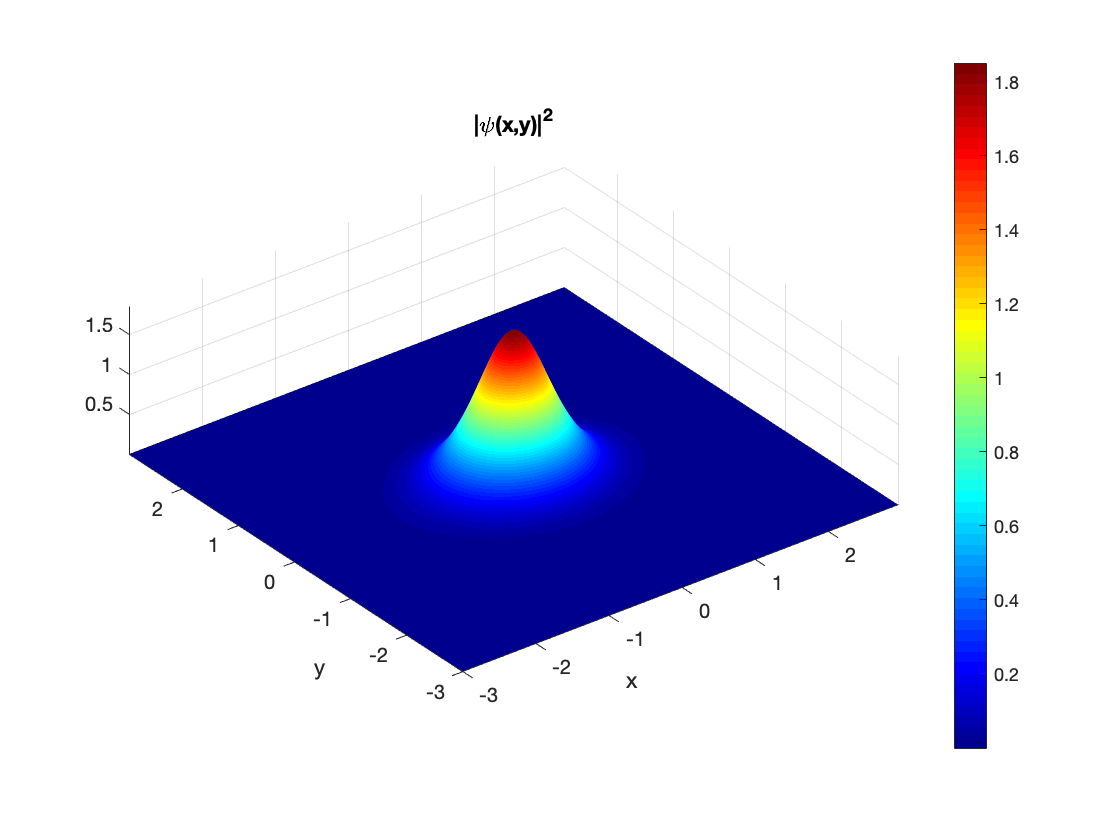}
	\caption{$\psi_0=1.5*Q_{\Omega,V}$\; (max $|\psi|^2\approx 1.85)$}    
\end{subfigure}
\caption{$|\psi|^2$ when $p=4$, $(\gamma_1, \gamma_2)=(1,2)$, $\Omega=0.5$ 
                   } 
                 \label{fig:p4-V12}
\end{figure}

\item  If turning off the rotation, i.e., $\Omega=0$, then Figure \ref{fig:p6-Om(0)V} shows that 
 in the isotropic case $\gamma_1=\gamma_2=1$,  $p=6$,  $\lambda=1$,  
 then blowup threshold $\psi_0=1.565Q_{\Omega,V}$; and there exists a bounded solution in $\mathscr{H}^1$  
   if $\psi_0= 1.56Q_{\Omega,V}$.    
However, in the anisotropic case for $V$, $\gamma_1=1, \gamma_2=2$, the blowup threshold is  
at level $\psi_0=1.395Q_{\Omega,V}$; and there exists a bounded solution in $\mathscr{H}^1$ 
if $\psi_0=1.39Q_{\Omega,V}$. 
The above results reveal that
higher order exponent $p$ and 
anisotropic property for the potential contribute  more to the wave collapse, which may make the 
system unstable at a lower level of mass.  It is of interest to observe that in the presence of rotation 
(Figure \ref{fig:p6-Om(half)V}),
the threshold constants $C$ remain the same in both isotropic and anisotropic cases, 
although $|\psi|^2$, the energies and chemical potentials grow at larger magnitude. 

\begin{figure}[H]
\centering
\begin{subfigure}[b]{.45\textwidth}
	\centering
	\includegraphics[width=\textwidth]{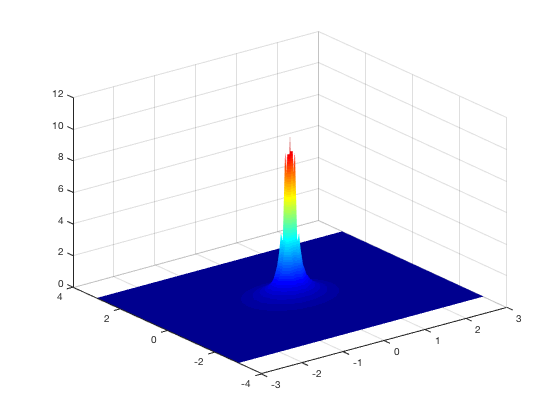}
	\caption{ $(\gamma_1,\gamma_2)=(1,1)$, $\psi_0=1.565*Q_{\Omega,V}$} 
\end{subfigure} 
\begin{subfigure}[b]{.45\textwidth}
       \centering
	\includegraphics[width=\textwidth]{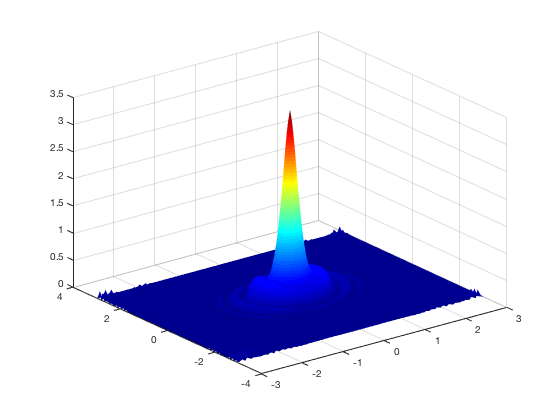}
	\caption{ $(\gamma_1,\gamma_2)=(1,2)$, $\psi_0=1.395*Q_{\Omega,V}$} 
\end{subfigure}
\caption{$|\psi|^2$ when $p=6$,  $V=\frac12(\gamma_1^2x^2+\gamma_2^2y^2)$, $\Omega=0$ 
                   } 
                 \label{fig:p6-Om(0)V}
\end{figure}
Notice that if $p=6$, then the behavior of wave-collapse is quite different than the case $p<5$.
 The modulus square of $\psi(t,x)$ first forms growing singularity. Then it quickly reduces to normal level 
but with large energy and  $\Vert \nabla \psi\Vert_2$ after collapsing time although it does not seem to admit proper self-similar  profile of energy concentration. 

\begin{figure}[H]
\centering
\begin{subfigure}[b]{.45\textwidth}
	\centering
	\includegraphics[width=\textwidth]{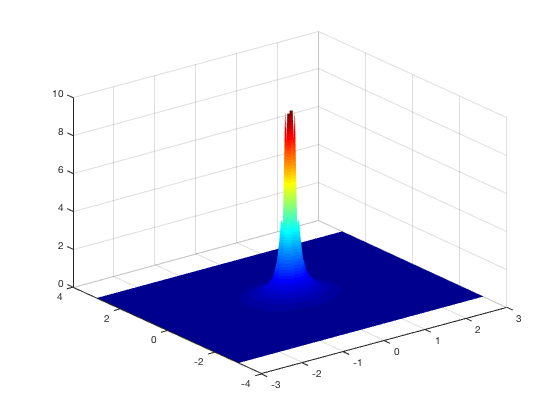}
	\caption{ $(\gamma_1,\gamma_2)=(1,1)$, $\psi_0=1.565*Q_{\Omega,V}$} 
\end{subfigure} 
\begin{subfigure}[b]{.45\textwidth}
       \centering
	\includegraphics[width=\textwidth]{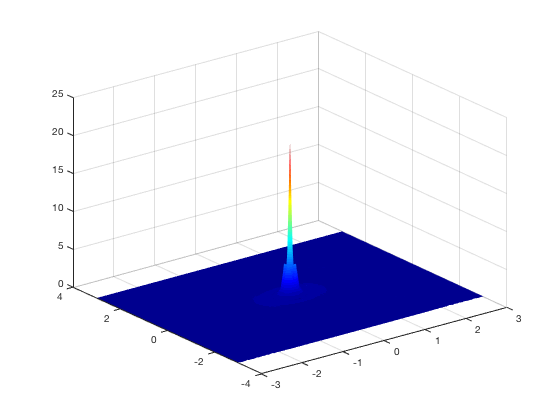}
	\caption{ $(\gamma_1,\gamma_2)=(1,2)$, $\psi_0=1.395*Q_{\Omega,V}$} 
\end{subfigure}
\caption{$|\psi|^2$ when $p=6$,  $V=\frac12(\gamma_1^2x^2+\gamma_2^2y^2)$, $\Omega=0.5$ 
                   } 
                 \label{fig:p6-Om(half)V}
\end{figure}
\end{enumerate}

\noindent{\bf Acknowledgement} 
 The authors thank the anonymous referee for helpful comments that have helped improve the presentation of the article.

\end{document}